\documentclass[12pt]{amsart}
\usepackage{amsfonts}
\usepackage{amsthm}
\usepackage{amsmath}
\usepackage{amssymb}
\usepackage{color}
\usepackage{dsfont}
\usepackage{setspace}
\numberwithin{equation}{section}
\theoremstyle{plain}
 \newtheorem{theorem}{Theorem}[section]
 \newtheorem{lemma}[theorem]{Lemma}
 
 \newtheorem{proposition}[theorem]{Proposition}

\theoremstyle{definition}
 \newtheorem{definition}[theorem]{Definition}
 \newtheorem{example}[theorem]{Example}
 \newtheorem{remark}[theorem]{Remark}

\newcommand{\bN}{\mathbb{N}}
\newcommand{\bZ}{\mathbb{Z}}
\newcommand{\bR}{\mathbb{R}}

\newcommand{\bC}{\mathbb{C}}

\newcommand{\supp}{\mbox{\rm supp }}

\newcommand{\bE}{\mathbb{E}}

\newcommand{\one}{\mathbf{1}}

\allowdisplaybreaks

\begin{document}

\vspace{5mm}
\begin{center}
{\bf
{\large
L\'{e}vy driven linear and semilinear stochastic partial differential equations}}

\vspace{5mm}

David Berger\\
\end{center}
\vspace{5mm}
The goal of this paper is twofold. In the first part we will study  L\'{e}vy white noise in different distributional spaces and solve equations of the type $p(D)s=q(D)\dot{L}$, where $p$ and $q$ are polynomials. Furthermore, we will study measurability of $s$ in Besov spaces. By using this result we will prove that stochastic partial differential equations of the form
\begin{align*}
p(D)u=g(\cdot,u)+\dot{L}
\end{align*}
have measurable solutions in weighted Besov spaces, where $p(D)$ is a partial differential operator in a certain class, $g:\bR^d\times \bC\to \bR$ satisfies some Lipschitz condition and $\dot{L}$ is a L\'{e}vy white noise.
\section{Introduction}
A stochastic process $X=(X_t)_{t\in\bR}$ is called a CARMA process, if $X$ is a solution of the (formal) stochastic differential equation 
\begin{align}\label{carmaone}
\sum\limits_{j=0}^m a_j \frac{d^{j}X(t)}{dt^{j}} =\sum\limits_{k=1}^n b_k \frac{d^{k}L(t)}{dt^{k}}
\end{align}
where $m,n\in\bN$, $a_j,b_k\in\bR$ for every $0\le j\le m$ and $0\le k\le n$ and $L$ is a L\'{e}vy process. Equation (\ref{carmaone}) can also be written as $a(D)X_t=b(D)L(t),$ where $a(z)=\sum\limits_{j=0}^m a_j z^j$ and $b(z)=\sum\limits_{j=1}^n b_j z^j$.  In \cite{Lindner} necessary and sufficient conditions on $L$ on the polynomials $a$ and $b$ were given such that there exists a strictly stationary solution of $(\ref{carmaone})$, namely it was shown that it is sufficient and necessary that $\bE \log^+(|L_1|)<\infty$. CARMA processes have many applications, see for example \cite{Garcia} and \cite{Brockwell3}.\\
For dimensions greater than 1, there exist more than one definition of a CARMA random field. Here, we will recall only the definition in the sense of Berger, see \cite{Berger}. For the other definitions see the two papers of Brockwell and Matsuda \cite{Brockwell} and Pham \cite{Pham}. In \cite{Berger} a CARMA random field $s$ is a stationary generalized stochastic process on the space of test functions $\mathcal{D}(\bR^d):=\mathcal{C}^{\infty}_c(\bR^d)$, which solves the equation
\begin{align}\label{carmad}
p(D)s=q(D)\dot{L},
\end{align}
where $p$ and $q$ are real polynomials in $d$-variables and $\dot{L}$ is L\'{e}vy white noise with characteristic triplet $(a,\gamma,\nu)$, where (\ref{carmad}) means that
\begin{align*}
\langle s,p(D)^*\varphi\rangle=\langle \dot{L},q(D)^*\varphi\rangle
\end{align*}
for every $\varphi\in\mathcal{D}(\bR^d)$, where $p(D)^*$ denotes the (formal) adjoint operator of $p(D)$. For the definition of stationary generalized processes and L\'{e}vy white noise see Section \ref{linearspde} or [\ref{Berger}, Definition 3.2]. It was shown that if the rational function $\frac{p}{q}$ has an holomorphic extension in a certain set and ${\int_{|r|>1}\log(|r|)^d\nu(dr)<}\nolinebreak\infty$, then there exists a stationary solution of (\ref{carmad}). The problem of the stationary generalized solution $s$ is that it may not have a random field representation and the question of uniqueness is open. Furthermore, as the regularity of $s$ is not well-understood in \cite{Berger}, it is not directly clear if one can solve more complex SPDEs than (\ref{carmad}). The goal of this paper is to tackle these problems and give some answers to these questions. We will show the existence of the L\'{e}vy white noise in the space of tempered ultradistributions and Fourier hyperfunctions defined as in \cite{Triebel1} and \cite{Kaneko1} and show that (\ref{carmad}) has solutions in the space of tempered (ultra-)distributions, Fourier hyperfunctions and Besov spaces under specific assumptions. Furthermore, we will analyze the semilinear equation
\begin{align}\label{semid}
p(D)s=g(\cdot,s)+\dot{L}
\end{align}
in certain weighted Besov spaces, where  $g:\bR^d\times \bC\to \bR$  is a sufficiently regular function. 
The above mentioned results can be found in Sections \ref{linearspde} and \ref{semilinearspdes}, where our main results are Theorem \ref{temperedultradistributions}, Theorem \ref{sfh} and Proposition \ref{lemma1}. In detail, in Section \ref{linearspde} we recall the definition of generalized stochastic processes and study (\ref{carmad}) in the three different spaces. In Section \ref{semilinearspdes} we study (\ref{semid}) in different Besov spaces.
\section{Notation and Preliminaries}\label{section0}
To fix notation, by $(\Omega,\mathcal{F})$ we denote a measurable space, where $\Omega$ is a set and $\mathcal{F}$ is a $\sigma$-algebra and by $L^0(\Omega,\mathcal{F},\mathbb{K})$ we denote all measurable functions $f:\Omega\to\mathbb{K}$ with respect to $\mathcal{F}$ where $\mathbb{K}=\bR, \bC$. In the case that $\mathcal{F}$ and $\mathbb{K}$ are clear from the context we set $L^0(\Omega)=L^0(\Omega,\mathcal{F},\mathbb{K})$. If we consider a probability space $(\Omega,\mathcal{F}, \mathcal{P})$, where $\mathcal{P}$ is a probability measure on $(\Omega,\mathcal{F})$, we say that a sequence $(f_n)_{n\in\bN}\subset L^0(\Omega)$ converges to $f$ in $L^0(\Omega)$ if $f_n$ converges in probability to $f$ with respect to the measure $\mathcal{P}$. In the case of $(\bR^d,\mathcal{B}(\bR^d))$ we denote by $\mathcal{B}(\bR^d)$ the Borel-$\sigma$-set on $\bR^d$. \\
We write $\bN=\{1,2,\dotso\}$, $\bN_0=\bN\cup \{0\}$ and $\bZ,\,\bR,\,\bC$ for the set of integers, real numbers and complex numbers, respectively. If $z\in \bC$, we denote by $\Im z$ and $\Re z$ the imaginary and the real part of $z$. The Euclidean norm is denoted by $\|\cdot\|$ and $r^+:=\max\{0,r\}$ for every $r\in\bR$ . By $C^\infty(\bR^d,\bC)$ we denote the set of all functions $\varphi :\bR^d\to \bC$ which are infinitely often differentiable. Furthermore, by $L^p(\bR^d, A)$ for $A\subseteq \bC$ and $0<p\le \infty$ we denote the set of all Borel-measurable functions $f:\bR^d \to A$ such that $\int_{\bR^d} |f(x)|^p\,\lambda^d(dx)<\infty$ for $0<p<\infty$ and $\mathop{\textrm{ess sup}}_{x\in\bR^d}|f(x)|<\infty$ for $p=\infty$, where $\lambda^d$ is the $d-$dimensional Lebesgue measure. We denote by $||f||_{L^p}=\left(\int_\bR |f(x)|^p\,\lambda(dx)\right)^{1/p}$ for $0<p<\infty$ and $\|f\|_{L^\infty}=\mathop{\textrm{ess sup}}_{\bR^d}|f|$ the $L^p$-(quasi-)norm for a measurable function $f$. We write $\langle x\rangle:=\left(1+\|x\|^2\right)^{1/2}$ and $\|f\|_{L^p(\bR^d,\rho)}:=\| \langle \cdot\rangle^{\rho}f\|_{L^p(\bR^d)}$ for $\rho\in \bR$.  Let $(a_k)_{k\in \bN_0} \subset \bC$ be a sequence and we set $$\|(a_k)_{k\in \bN_0}\|_{l^q}:=\left(\sum\limits_{k\in \bN_0} |a_k|^q\right)^{\frac{1}{q}}$$
for $0<q< \infty$. For $q=\infty$ the norm is given by $\|(a_k)_{k\in \bN_0}\|=\sup_{k \in \bN_0}|a_k|$. By $d_f$ we denote the distribution function of $f:\bR^d\to \bC$, which means that 
\begin{align}
d_f(\alpha):=\lambda^d(\{x\in \bR^d: |f(x)|>\alpha\}),\, \alpha\ge 0.
\end{align}
The space $\mathcal{D}(\bR^d)$ denotes the set of all infinitely differentiable functions $f:\bR^d\to\bR$ with compact support with its usual topology (e.g. [\ref{Fageot}, Section 2.1]), where we denote the support of $f$ by $\supp f$. The topological dual space of $\mathcal{D}(\bR^d)$ will be denoted by $\mathcal{D}'(\bR^d)$, where an element  $u\in\mathcal{D}'(\bR^d)$ is called a distribution. The space $\mathcal{S}(\bR^d)$ denotes the Schwartz space equipped with its usual topology, see [\ref{Dalang1}, Section 1, p. 4391] and $\mathcal{S}'(\bR^d)$ its topological dual with its strong topology. We sometimes write $\mathcal{S}$ and $\mathcal{S}'$, if the dimension is clear. We will write $\langle u,\varphi\rangle:=u(\varphi)$ for $\varphi\in\mathcal{D}(\bR^d)$ (or $\mathcal{S}(\bR^d)$) and $u\in \mathcal{D}'(\bR^d)$ (or $\mathcal{S}'(\bR^d)$). We say that a function $a:Y\to\bR$ from some function space $Y$ acts as a Fourier multiplier for some function space $X$ to a function space $R$ with well-defined Fourier transform $\mathcal{F}$ if $a:X\to R$ is defined by $a(u):=\mathcal{F}^{-1}(a\mathcal{F}u)$, where $(a\mathcal{F}(u))(t)=a(t)\mathcal{F}(u)(t)$ such that the inverse Fourier transform $\mathcal{F}^{-1}$ is well-defined. For a function $f\in L^1(\bR^d,\bC^d)$ we set $\mathcal{F} f(x)=\int\limits_{\bR^d} e^{-i \langle z, x\rangle}f(z)\lambda^d(dz)$ and the $L^2$-Fourier transform likewise. A polynomial $p$ is a function given by $p(z)=\sum\limits_{|\alpha|\le m}p_{\alpha}z^{\alpha}$, $\alpha\in \bN^d_0$, $m\in\bN$, $z^{\alpha}=z_1^{\alpha_1}\dotso z_d^{\alpha_d}$ and $|\alpha|:=\alpha_1+\dotso+\alpha_d$. 
We set $D^{\alpha}=\partial_{x_1}^{\alpha_1}\dotso\partial_{x_d}^{\alpha_d}$ for $\alpha \in \bN^d_0$. We denote by $A^*$ the adjoint of the operator $A$.\\ We introduce weighted Besov spaces and follow \cite{Triebel3}.
Let $\varphi_0\in \mathcal{S}(\bR^d)$ such that $\varphi_0(x)=1$ if $\|x\|\le  1$ and $\varphi_0(x)=0$ if $\|x\|\ge 3/2,$ and we set
\begin{align*}
\varphi_k(x)=\varphi_0(2^{-k}x)-\varphi_0(2^{-k+1}x), \,x\in\bR^d,\, k\in\bN.
\end{align*}
As $\sum\limits_{k=0}^\infty \varphi_k(x)=1$ for all $x\in\bR^d$, it  is clear that $(\varphi_k)_{k\in \bN_0}$ is a dyadic decomposition of unity in $\bR^d$. We set
\begin{align*}
\Delta_kf:=\mathcal{F}^{-1} \varphi_k \mathcal{F}f
\end{align*}
for every $f\in \mathcal{S}'(\bR^d)$. Observe that this object is a well-defined function, which can be evaluated pointwise, see [\ref{Triebel1}, Remark 1, p. 37].
A weighted Besov space ${B}^l_{r,t}(\bR^d,\rho)$ is a subspace of $\mathcal{S}'(\bR^d)$ which is characterized by four parameters $l,\rho\in\bR$ and $r,t>0$, where $f\in {B}^l_{r,t}(\bR^d,\rho)$ if and only if
\begin{align*}
\|f\|_{B^l_{r,t}(\bR^d,\rho)}:=\| (2^{lk}\|\Delta_k f\|_{L^r(\bR^d,\rho)})_{k\in \bN_0}\|_{l^t}<\infty.
\end{align*}
For $r=t=2$ and $l>0$ we identify the weighted Sobolev space $W^l_2(\bR^d,\rho)$ with $B_{2,2}^l(\bR^d,\rho)$, i.e. there exists a continuous and bijective mapping $\xi$ from $B_{2,2}^l(\bR^d,\rho)$ to $W^l_2(\bR^d,\rho)$ such that for all $f\in B_{2,2}^l(\bR^d,\rho)$ 
\begin{align}\label{mapping}
f(\varphi)=\int_{\bR^d} \xi(f)(x)\varphi(x) \lambda^d(dx)\textrm{ for all }\varphi \in \mathcal{S},
\end{align}
see [\ref{Triebel1}, Theorem 2.5.6, p. 88]. Moreover, $\xi$ is also continuous from $B_{r,r}^l(\bR^d,\rho)$ to $L^r(\bR^d,\rho)$ for $l>0$, $r\ge 2$. From now on we write for $\xi(f)$ simply $f$.
\\An interesting property of the Besov spaces are their embeddings, which are described as follows:
\begin{proposition}[see {[\ref{Fageot1}, Proposition 3, p. 1605]}]\label{embeddings}
Let $p_0,p_1\in (0,\infty]$ and $\tau_0, \tau_1, \rho_0,\rho_1\in \bR$ with $\tau_0\ge \tau_1$. It holds that $B_{p_0,p_0}^{\tau_0}(\bR^d, \rho_0)$ is continuously embedded in  $B_{p_1,p_1}^{\tau_1}(\bR^d,\rho_1)$ if $\tau_0-\tau_1\ge\frac{d}{p_0}-\frac{d}{p_1}$, $p_1\ge p_0$ and $\rho_0\ge\rho_1$. If the inequalities are strict, the embeddings are compact.
\end{proposition}
\section{Linear stochastic partial differential equations in the spaces of tempered distributions, tempered ultradistributions, Fourier hyperfunctions and Besov spaces with polynomial weights}\label{linearspde}
At first we give a short introduction to generalized processes on more general distributional space $A'$, which is the dual space of a suitable function space $A$. For example, $A$ can be the Schwartz space $\mathcal{S}(\bR^d)$, the space of test functions $\mathcal{D}(\bR^d)$ or even weighted Besov spaces $B_{p,q}^{s}(\bR^d, \rho)$ for suitable $s,\,p$ and $\rho$.
\begin{definition}[see {\cite{Fageot}, Definition 2.1}]
An \emph{$A'$-valued generalized random process} $s$ is a measurable mapping from $(\Omega,\mathcal{F})$ to a distributional space $(A', C(A'))$, where $C(A')$ denotes the $\sigma$-field generated by the cylindrical sets
\begin{align*}
\{u\in A': \langle u,\varphi_j\rangle \in B\textrm{ for every }j=1,\dotso,n\}
\end{align*}
for every $\varphi_1,\dotso,\varphi_n\in A$ and $B\in\mathcal{B}(\bR)$.
\end{definition}
Since in our cases under consideration $A'$ will be nuclear or even a Hilbert space, it follows from [\ref{Ito}, p.6] that $C(A')$ is equal to the $\sigma-$algebra $B^*(A')$ generated by the weak-$\ast$-topology in $A'$.\\
The probability law of a generalized random process $s$ is given by
\begin{align*}
\mathcal{P}_s(B):=\mathcal{P}(s\in B)
\end{align*}
for $B\in\mathcal{B}^*(A')$. The characteristic functional $\widehat{\mathcal{P}}_s$ is then defined by
\begin{align*}
\widehat{\mathcal{P}}_s(\varphi)=\int\limits_{A'}\exp(i\langle u,\varphi\rangle)d\mathcal{P}_s(u), \, \varphi\in A.
\end{align*}
We will work with L\'{e}vy white noise, which is a generalized random process, where the characteristic functional satisfies a L\'{e}vy-Khintchine representation. 
\begin{definition}
A \emph{L\'{e}vy white noise} $\dot{L}$ on $A'$ is an $A'$-valued generalized random process, where the characteristic functional is given by
\begin{align*}
\widehat{\mathcal{P}}_{\dot L}(\varphi)=\exp\left(\,\,\int\limits_{\bR^d} \psi(\varphi(x))\lambda^d(dx)\right)
\end{align*}
for every $\varphi \in A$, where $\psi:\bR\to \bC$ is given by
\begin{align*}
\psi(z)=i\gamma z-\frac{1}{2}az^2+\int\limits_{\bR}(e^{ixz}-1-ixz\one_{|x|\le1})\nu(dx)
\end{align*}
where $a\in\bR^+$, $\gamma\in\bR$ and $\nu$ is a \emph{L\'{e}vy-measure}, i.e. a measure such that $\nu(\{0\})=0$ and $$\int\limits_{\bR} \min(1,x^2)\nu(dx)<\infty.$$ We say that $\dot{L}$ has the \emph{characteristic triplet} $(a,\gamma,\nu)$.
\end{definition}
The L\'{e}vy white noise is stationary in the following sense:
\begin{definition}
A generalized random process $s$ is called \emph{stationary} if for every $t\in\bR^d$, $s(\cdot+t)$ has the same law as $s$. Here, $s(\cdot+t)$ is defined by
$$\langle s(\cdot+t),\varphi\rangle:=\langle s, \varphi(\cdot-t)\rangle \textrm{ for every }\varphi\in A.$$
\end{definition}
It is well-known that a L\'{e}vy white noise $\dot{L}$ on the space of tempered distributions $\mathcal{S}'$ with characteristic triplet $(a,\gamma,\nu)$ exists if and only if there exists an $\varepsilon>0$ such that $\int\limits_{|r|>1}|r|^{\varepsilon}\nu(dr)<\infty$, see [\ref{Dalang1}, Theorem 3.13, p. 4412]. As the space of tempered distributions is too small for many cases of the L\'{e}vy white noise, we will construct the L\'{e}vy white noise in another distributional space. We discuss the existence of the L\'{e}vy white noise in the space of tempered ultradistribution. The space of tempered ultradistributions is very similar to the space of tempered distributions, especially the space $\mathcal{S}'_{\omega}$ is nuclear, which allows us to use the Bochner-Minlos Theorem. Moreover, by similar arguments we construct L\'{e}vy white noise in the space of Fourier hyperfunctions.
Furthermore, we will discuss in the spirit of [\ref{Berger}, Theorem 4.3] the solvability of the equations 
\begin{align}\label{carma}
p(D)s=q(D)\dot{L}
\end{align}
 in the space of tempered distributions, tempered ultradistributions and Fourier hyperfunctions, where $p(z)=\sum_{|\alpha|\le n} p_{\alpha}z^{\alpha}$ and  $q(z)=\sum_{|\alpha|\le m} q_{\alpha}z^{\alpha}$ are real multivariate polynomials. Moreover, we will study (\ref{carma}) also for L\'{e}vy white noise in Besov spaces, as these results are needed in Section \ref{semilinearspdes}  for more complex (nonlinear) stochastic partial differential equations. We start with an existence result on the space of tempered distributions of (\ref{carma}). Observe that $p(D)^*=p(-D)$ and $q(D)^*=q(-D)$.
\begin{proposition}\label{prop1}
Let $\dot{L}$ be a L\'{e}vy white noise on the space of tempered distributions $\mathcal{S}'$. Let $p$ and $q$ be two polynomials such that there exists two polynomials $h$ and $l$ such that  $\frac{q(i\cdot)}{p(i\cdot)}=\frac{h(i\cdot)}{l(i\cdot)}$ on $\bR^d$ and $l$ has no zeroes on $i\bR^d$. Then there exists a generalized process $s$ on the space of tempered distributions solving (\ref{carma}), i.e. it holds that
\begin{align}\label{s}
\langle s,p(D)^*\varphi\rangle=\langle\dot{L},q(D)^*\varphi\rangle
\end{align}
for all $\varphi\in \mathcal{S}(\bR^d)$, which is stationary. If $p(iz)\neq 0$ for all $z\in\bR^d$, then the solution $s$ is unique.
\end{proposition}
\begin{proof}
We observe by [\ref{Hormander}, Lemma 2] that $\varphi\mapsto\mathcal{F}^{-1}\left( \frac{q(i\cdot)}{p(i\cdot)}\mathcal{F}\varphi\right) $ defines a continuous operator from $\mathcal{S}(\bR^d)$ to $\mathcal{S}(\bR^d)$ and define 
\begin{align}\label{construction}
\langle s,\varphi\rangle:=\langle \dot{L}, \mathcal{F}^{-1}\left( \frac{q(i\cdot)}{p(i\cdot)}\mathcal{F}\varphi\right)\rangle.
\end{align}
We conclude that $s$ defines a generalized process on the space of tempered distributions. That it solves (\ref{s}) follows easily by $\mathcal{F}p(-D)\varphi=p(-i\cdot)\mathcal{F}\varphi$ for every $\varphi \in\mathcal{S}$ and the stationarity of $s$ follows from that of $\dot{L}$.\\
Now let $u$ be another solution of the equation $(\ref{carma})$. One observes that \begin{align}\label{null}\langle p(D)(s-u),\varphi\rangle=0\end{align} for every $\varphi\in\mathcal{S}$. In the case that $p(iz)\neq 0$ for all $z\in\bR^d$ it is known that only the null-solution satisfies equation (\ref{null}), see [\ref{Ortner}, Proposition 2.4.1, p. 152], so we conclude $s=u$.
\end{proof}
Our second distribution space is the space of tempered ultradistributions. For a detailed introduction to these spaces see \cite{Triebel1}. We recall the definition.
\begin{definition}\label{ultra}
Let $\omega:\bR^d\to \bR$ be a real-valued function such that $\omega(x)=\sigma(\|x\|)$, where $\sigma(t)$ is an increasing continuous concave function on $[0,\infty)$ with
\begin{align*}
&\sigma(0)=0,\\
&\int\limits_{0}^\infty \frac{\sigma(t)}{1+t^2}\lambda^1(dt)<\infty,\\
&\sigma(t)\ge c+m\log(1+t)\textrm{ if }t\ge 0
\end{align*}
for some $c\in\bR$ and $m>0$. Then the space $\mathcal{S}_\omega$ is the set of all infinitely differentiable functions $\varphi: \bR^d\to\bC$ such that
\begin{align*}
&p_{\alpha,\eta}(\varphi):=\sup_{x\in\bR^d} e^{\eta \omega(x)} \|D^\alpha \varphi(x)\|<\infty,\\
&\pi_{\alpha,\eta}(\varphi):=\sup_{x\in\bR^d} e^{\eta \omega(x)} \|D^\alpha (\mathcal{F}\varphi)(x)\|<\infty,
\end{align*}
for every multi-index $\alpha$ and every $\eta>0$. The space is equipped with its seminorms given above and its topological dual $\mathcal{S}'_{\omega}$ is called the \emph{space of tempered ultradistributions}.
\end{definition}
We denote by $\omega^{\rightarrow}(\alpha):=\sup\{ x\in[0,\infty):\omega(xe_1)<\alpha\}$ for $\alpha\in (0,\infty)$, where $e_1$ is the unit vector $(1,0\dotso,0)$.  \\
We split a function $\varphi \in C^\infty(\bR^d,\bC)$ in its real and imaginary part and prove the existence of a L\'{e}vy white noise on 
$$\mathcal{S}_{\omega}^{real}=\mathcal{S}_\omega\cap \{\varphi:\bR^d\to \bC:\varphi(x)\in \bR\textrm{ for all }x\in\bR^d\}.$$
Observe that $\mathcal{S}_{\omega}^{real}$ equipped with the topology of $\mathcal{S}_\omega$ is closed and therefore nuclear.
We then set $\langle \dot{L},\varphi\rangle:=\langle \dot{L},\Re \varphi\rangle+i\langle \dot{L},\Im \varphi\rangle$ which defines the L\'{e}vy white noise on $\mathcal{S}_\omega$. 
\begin{theorem}\label{levyultra}
Let $(a,\gamma,\nu)$ be a characteristic triplet and $\omega$ be a function defined as in Definition \ref{ultra}. If
\begin{align*}
\int\limits_{|r|>1} |r|\int\limits_{0}^{1/|r|}\omega^{\rightarrow}(c\log(|\alpha|^{-1}))^d\lambda^1(d\alpha) \nu(dr)<\infty
\end{align*}
for some $c \in (0,\infty)$, then there exists a L\'{e}vy white noise $\dot{L}:(\Omega,\mathcal{F})\to(\mathcal{S}_{\omega}',C(\mathcal{S}_{\omega}'))$ with characteristic triplet $(a,\gamma,\nu)$.
\end{theorem}
\begin{proof}
We need to show that the function
\begin{align*}
\mathcal{P}(\varphi):=\exp\left(\,\,\int\limits_{\bR^d}\left(i\gamma \varphi(u)-\frac{1}{2}a\varphi(u)^2+\int\limits_{\bR}(e^{ix\varphi(u)}-1-ix\varphi(u)\one_{|x|\le1})\nu(dx)\lambda^d(du)\right)\right)
\end{align*}
defines a continuous and positive-definite mapping on $\mathcal{S}_\omega^{real}$ and $\mathcal{P}(0)=1$. Then we conclude  by the Bochner-Minlos Theorem [\ref{Fageot2}, Theorem 1, p. 1186] that there exists a L\'{e}vy white noise in $(\mathcal{S}'_\omega, C(\mathcal{S}'_\omega))$ with characteristic triplet $(a,\gamma,\nu)$.\\
That $\mathcal{P}(0)=1$ is trivial, so we start with the continuity. Therefore, let $\rho(x):= \exp(-\eta \omega(x))$ with $\eta>0$. We see that
\begin{align}
\nonumber d_\rho(\alpha):&=\lambda^d(\{x\in\bR^d: \rho(x)>\alpha\})\\
\nonumber&=\lambda^d(\{x\in \bR^d: \omega(x)<\log(\alpha^{-1})/\eta\})\\
&=c_d  \omega^{\rightarrow}(\log(|\alpha|^{-1})/\eta)^d \label{equality1}
\end{align}
for some constant $c_d>0$. Now let $(\varphi_n)_{n\in\bN}$ be a sequence in $\mathcal{S}_\omega^{real}$ such that $\varphi_n\to 0$ in $\mathcal{S}_\omega$, which implies that $\sup_{x\in\bR^d} e^{\eta\omega(x)}|\varphi_n(x)|\to 0$ for $n\to\infty$ for all $\eta>0$ or equivalently $|\varphi_n(x)|\le c_n e^{-\eta\omega(x)}$ for all $x\in\bR^d$ for a sequence $(c_n)_{n\in\bN}\subset [0,\infty)$ converging to $0$ for $n\to\infty$. We conclude by (\ref{equality1}) and Lebesgue's dominated convergence theorem that
\begin{align*}
\int\limits_{|r|>1} |r|\int\limits_{0}^{1/|r|}d_{\varphi_n}(\alpha)\lambda^1(d\alpha)\nu(dr)\to 0,\,n\to\infty.
\end{align*} 
Now by similar arguments as in [\ref{Berger}, Theorem 3.4] we see that $\mathcal{P}$ is continuous on $\mathcal{S}_\omega^{real}$. That $\mathcal{P}$ is positive definite follows by a denseness argument similar to [\ref{Fageot2}, Proposition 2, p. 1187].
\end{proof}
We give two examples and obtain for a special weight $\omega$ the space $\mathcal{S}(\bR^d)$.
\begin{example}
Let $\omega(x)=m\log(1+\|x\|)$ for $m>0$. It is well-known that $\mathcal{S}_\omega=\mathcal{S}$, see [\ref{Triebel1}, Remark 4, p. 246], and we see that
$$\omega^{\rightarrow}(\alpha)=e^{\alpha/m}-1$$
for all $\alpha>0$ and we obtain for $c\in (0,m)$ and $\alpha\in (0,1)$
 $$\omega^{\rightarrow}(c\log(|\alpha|^{-1}))=\omega^{\rightarrow}(\log(|\alpha^{-c}|))\le  \alpha^{-c/m}.$$ As $c\in (0,1)$ is arbitrary, we conclude from Theorem \ref{levyultra} that if $\int_{|r|>1}|r|^{\varepsilon}\nu(dr)<\infty$ for some $\varepsilon>0$ there exists a L\'{e}vy white noise with characteristic triplet $(a,\gamma,\nu)$ on $\mathcal{S}'_{\omega}=\mathcal{S}'$, thus recovering the sufficient condition of [\ref{Dalang1}, Theorem 3.13, p. 4412].
\end{example}
\begin{example}
Let $\omega(x)=\|x\|^{\beta}$ for $0<\beta<1$. It is easily seen that $\omega$ satisfies the assumptions of Definition \ref{ultra} and furthermore,
\begin{align*}
\omega^{\rightarrow}(\alpha)=\alpha^{1/\beta}
\end{align*}
for $\alpha\in (0,\infty)$. We conclude from Theorem \ref{levyultra} that a L\'{e}vy white noise with characteristic triplet $(a,\gamma,\nu)$ exists on $\mathcal{S}'_\omega$ if $$\int_{|r|>1} (\log(|r|))^{d/\beta}\nu(dr).$$
\end{example}
In the next step we will analyze equation (\ref{carma}) in the space of tempered ultradistributions. We will obtain similar results as in Proposition \ref{prop1}. 
\begin{theorem}\label{temperedultradistributions}
Let $p,q$ be two real polynomials and assume that the rational function $q(i\cdot)/p(i\cdot)$ has a holomorphic extension in a strip $\{z\in\bC^d: \|\Im z\|<\varepsilon\}$ for some $\varepsilon>0$. Furthermore, let $\omega$ be as in Definition \ref{ultra} and $\dot{L}$ be a L\'{e}vy white noise on the space of tempered ultradistribution $\mathcal{S}'_\omega$ under the conditions of Theorem \ref{levyultra}. Then there exists a generalized stationary process $s$ in the space of tempered ultradistributions $\mathcal{S}_\omega'$ such that
\begin{align*}
p(D)s=q(D)\dot{L}.
\end{align*}
Moreover, if $p(i\cdot)$ has no zeroes in the strip, then the solution is unique.
\end{theorem}
\begin{proof}
Observe that for every $c>0$ there exists an $n\in\bN$ such that $\omega(x)\le n+ c\| x\|$ for all $x\in\bR^d$, otherwise, the assumption of Definition \ref{ultra} can not hold true. Now choose $\alpha\in\bN$ such that 
$$
\sup_{\|\eta\|\le \delta} \left\|  \frac{q(-i\cdot+\eta)}{p(-i\cdot+\eta)\psi(\cdot+i\eta)}\right\|_{L^1(\bR^d)}<\infty
$$
for some $0<\delta<\varepsilon$, where $\psi(z):=(1+\sum\limits_{j=1}^d z_j^2)^{\alpha}$. It is not immediately clear that such an $\alpha\in \bN$ exists,  but by a similar argument as in  the proof of [\ref{Hormander}, Lemma 2, p. 557] we conclude that such an $\alpha$ exists. We define $G:=\mathcal{F}^{-1}\frac{q(-i\cdot)}{p(-i\cdot)\psi(\cdot)}$ and we observe that there exists a constant $C>0$ such that $|G(x)|\le C\exp\left(-\frac{\delta}{2}\|x\|\right)$ for all $x\in \bR^d$, see [\ref{Reed}, Theorem IX.14, p. 18]. One infers by the subadditivity of $\omega$ (see [\ref{Triebel1}, Remark 2, p. 246]) for $\varphi\in\mathcal{S}_\omega$ that
\begin{align}
\nonumber |G\ast\varphi(x)|&\le \int\limits_{\bR^d}|G(y) \varphi(x-y)|\lambda^d(dy)\\
\label{sufficient}&\le p_{0,\eta}(\varphi)\int\limits_{\bR^d}|G(y) e^{-\eta \omega(y-x)}|\lambda^d(dy)\le p_{0,\eta}(\varphi)e^{-\eta \omega(x)}\int\limits_{\bR^d}|G(y) e^{\eta \omega(y)}|\lambda^d(dy)
\end{align}
for every $\eta>0$. We conclude that $|G\ast \varphi(x)|\le C e^{-\eta \omega(x)}$ for some constant $C>0$ and as $D^{\alpha}\varphi \in \mathcal{S}_{\omega}$, we conclude that $$\sup_{x\in\bR^d}e^{\eta \omega(x)}\|D^{\alpha}G\ast \varphi(x)\|=\sup_{x\in\bR^d}e^{\eta \omega(x)}\|G\ast D^\alpha\varphi(x)\|<\infty$$
for every $\eta>0$. Moreover, for the Fourier-transform of $G\ast\varphi$ it is easy to see that
\begin{align}\label{fourierultra}
\sup_{z\in\bR^d}e^{\eta\omega(z)}\|D^{\alpha} \mathcal{F}(G\ast\varphi)(z)\|=\sup_{z\in\bR^d}e^{\eta\omega(z)}\|\sum\limits_{|\beta|\le |\alpha|} p_\beta(z) D^{\beta}\mathcal{F}\varphi(z)\|<\infty,
\end{align}
which follows by  [\ref{Hormander}, Lemma 2, p. 557] for every $\eta>0$, where $p_\beta$ are rational functions well-defined on $\{z\in\bC^d: \|\Im z\|<\varepsilon\}$ for every $|\beta|\le |\alpha|$. So by similar estimates as in (\ref{sufficient}) and (\ref{fourierultra}) one sees that
\begin{align}\label{operatorultra}
\varphi\mapsto\mathcal{F}^{-1} \frac{q(-i\cdot)}{p(-i\cdot)\psi(\cdot)} \mathcal{F} \psi(\cdot)\varphi
\end{align}
defines a continuous operator from $\mathcal{S}_{\omega}$ to $\mathcal{S}_\omega$. Hence
\begin{align*}
\langle s,\varphi\rangle=\langle \dot{L}, \mathcal{F}^{-1} \frac{q(-i\cdot)}{p(-i\cdot)\psi(\cdot)} \mathcal{F} \psi(\cdot)\varphi\rangle,
\end{align*}
defines a generalized random process $s:(\Omega,\mathcal{F})\to (\mathcal{S}'_\omega,C(\mathcal{S}'_{\omega}))$. That is solves (\ref{carma}) follows as in [\ref{Berger}, Theorem 4.3].\\
The uniqueness of the solution $s$ follows by the proof of [\ref{Kaneko}, Proposition 2.2].
\end{proof}
We have shown so far that for every L\'{e}vy white noise $\dot{L}$ with characteristic triplet $(a,\gamma,\nu)$ living in $S'_{\omega}$ for every $\omega$ satisfying the assumptions of Definition \ref{ultra} there exists a unique solution $s$ of the equation
\begin{align*}
p(D)s=q(D)\dot{L},
\end{align*}
if $p(i\xi)$ has no zeroes in a strip around $\bR^d$. In \cite{Berger}, Theorem 3.4 we have seen that we obtain a solution $s$ in the space of distributions $\mathcal{D}'$ in the case that
\begin{align}\label{cond1}
\int_{|r|>1} \log(|r|)^d\nu(dr)<\infty,
\end{align}
but it seems difficult to find such $\omega$ such that $\dot{L}$ would be living in $S'_{\omega}$ if the L\'{e}vy white noise satisfies (\ref{cond1}), as (\ref{cond1}) does not imply the condition of Theorem \ref{levyultra} for any suitable weight functions $\omega$. Therefore, we use another more suitable space for the analysis of (\ref{carma}) under the assumption of (\ref{cond1}), the space of analytic functions with rapid decay and its topological dual, the Fourier hyperspace.
\begin{definition}
The space $\mathcal{P}_*$ consists of all functions $\varphi\in C^\infty(\bR^d,\bC)$ which have an analytic continuation on a strip 
\begin{align*}
A_\delta:=\{z\in \bC^d:\, \|\Im z\|<\delta\}
\end{align*}
for some $\delta>0$ and it holds that
\begin{align}\label{se}
\sup_{z\in A_l} |\exp((\delta-\varepsilon)\|z\|)\varphi(z)|<\infty
\end{align}
for every $0,\varepsilon,l<\delta$.
The space $\mathcal{P}_*$ is nuclear with its inductive topology, i.e. a sequence $(\varphi_n)_{n\in\bN}\subset \mathcal{P}_*$ converges to $0$ if and only if there exists a $\delta>0$ such that $\varphi_n$ has an analytic continuation in $A_\delta$ for every $n\in \bN$ and 
\begin{align*}
\sup_{z\in A_{\delta/2}} |\exp(\delta/2\|z\|)\varphi_n(z)|\to 0 \textrm{ for }n\to\infty,
\end{align*}
see [\ref{Kaneko1}, p. 408]. We denote by $\mathcal{Q}$ its topological dual and call it the \emph{space of Fourier hyperfunctions}. 
\end{definition}
We show first that there exists a L\'{e}vy white noise on $\dot{L}$ with characteristic triplet $(a,\gamma,\nu)$ on $\mathcal{Q}$ if (\ref{cond1}) holds. Observe that we split a function $\varphi \in C^\infty(\bR^d,\bC)$ in its real and imaginary part and prove on each part separately the existence of the L\'{e}vy white noise. Then $\langle \dot{L},\varphi\rangle:=\langle \dot{L},\Re \varphi\rangle+i\langle \dot{L},\Im \varphi\rangle$.  
\begin{proposition}\label{fh}
Let $(a,\gamma,\nu)$ be a characteristic triplet such that (\ref{cond1}) holds true. Then there exists a L\'{e}vy white noise on $(\mathcal{Q},C(\mathcal{Q}))$.
\end{proposition}
\begin{proof}
The proof is very similar to that of Theorem \ref{levyultra}. At first we observe for $\rho(x):=D\exp(-\delta\|x\|)$ for some $D,\delta>0$ that
\begin{align*}
d_{\rho}(\alpha)&=\lambda^d(\{x\in \bR^d:\rho(x)>\alpha\})\\
&=\lambda^d(\{x\in \bR^d:\|x\|< -\frac{1}{\delta}\log(\alpha/D)\})\\
&=-C \log(\alpha/D)^{d}
\end{align*} 
 for some constant $C>0$ for all $\alpha< D$. So  for every sequence $(\varphi_n)_{n\in \bN}\subset \mathcal{P}_*$ convgerging to $0$ we obtain similar to \cite{Berger}, Example 3.8 that 
\begin{align*}
\int_{|r|>1}|r|\int_{0}^{1/|r|} d_{\varphi_n}(\alpha)\lambda^1(d\alpha)\nu(dr)\to 0,\,n\to\infty.
\end{align*}
By following the same argumentation as in the proof of Theorem \ref{levyultra} one infers that there exists a L\'{e}vy white noise $\dot{L}$ on $\mathcal{Q}$.
\end{proof} 
As a final step we prove the unique solvability of equation (\ref{carma}) in $\mathcal{Q}$.
\begin{theorem}\label{sfh}
Let $\dot{L}$ be a L\'{e}vy white noise on $\mathcal{Q}$. Assume that  $p,q$ be two real polynomials such that the rational function $q(i\cdot)/p(i\cdot)$ has a holomorphic extension in a strip $\{z\in\bC^d: \|\Im z\|<\varepsilon\}$ for some $\varepsilon>0$.Then there exists a generalized stationary process $s$ in $\mathcal{Q}$ such that
\begin{align*}
p(D)s=q(D)\dot{L}.
\end{align*}
Moreover, if $p$ has no zeroes in the strip, than the solution is unique.
\end{theorem}
\begin{proof}
The uniqueness when $p$ has no zeroes on the strip follows in the same manner as in Proposition \ref{prop1}  by the proof of [\ref{Kaneko}, Proposition 2.2].\\
For the existence of the stationary solution $s$, let $G$, $\psi$ and $\alpha$ be as in the proof of Theorem \ref{temperedultradistributions}. Since $(\mathcal{F}^{-1}\psi(\cdot)\mathcal{F} \varphi)=(1-\Delta)^{\alpha}\varphi$, it follows similarity to the proof of [\ref{Berger}, Theorem 4.3] that it is sufficient to show that
$$T:\mathcal{P}_*\to \mathcal{P}_*,\,\varphi\mapsto G\ast (1-\Delta)^{\alpha}\varphi$$
is continuous. To see this let $(\varphi_n)_{n\in \bN}\subset \mathcal{P}_*$ be converging to $0$, i.e. there exists a $\delta>0$ such that $\varphi_n$ has an analytic continuation in $A_\delta$ for every $n \in \bN$ and $\sup_{z\in A_\delta}|\exp(\delta \|z\|) \varphi_n(z)|\to 0$ for $n\to \infty$. Then it holds by Cauchy's integral formula for derivatives that $(1-\Delta)^{\alpha}\varphi_n\in \mathcal{P}_*$ for every $n\in\bN$ and $(1-\Delta)^{\alpha}\varphi_n\to 0$ for $n\to\infty$ in $\mathcal{P}_*$. So it is sufficient to show that $\tilde{T} :\mathcal{P}_*\to \mathcal{P}_*$ defined by $\tilde{T}(\varphi):=G\ast \varphi$ is continuous. This follows easily by the same method as in the proof of Theorem \ref{temperedultradistributions}. Therefore we obtain a mapping $s:\Omega\to \mathcal{Q}$ defined by $s(\varphi):=\langle\dot{L}, G\ast (1-\Delta)^{\alpha}\varphi\rangle$ for every $\varphi \in \mathcal{P}_*$, which solves (\ref{carma}) and is stationary.
\end{proof}
As a L\'{e}vy white noise $\dot{L}$ on $\mathcal{S}'(\bR^d)$ with characteristic triplet $(a,\gamma,\nu)$ such that $\int_{|r|>1} |r|^{\varepsilon}\nu(dr)<\infty$ lives in $\mathcal{S}'(\bR^d)$, it is only natural to ask if the L\'{e}vy white noise can be constructed on certain negative Sobolev spaces with some weights. In [\ref{Fageot1}] it was shown that $P(\dot{L}\in W^{\tau}_{2}(\bR^d,\rho))=1$ for $\tau<-d/2$ and $\rho<-d/\min\{\varepsilon,2\}$. Indeed, there exists even a generalized process on $(W^{\tau}_2(\bR^d,\rho), \mathcal{B}^*(W^{\tau}_2(\bR^d,\rho)))$ which follows by [\ref{Ito}, Theorem 1.2.4, p.6]. Moreover,  $\dot{L}$ can be seen as a random variable on the space $(W^{\tau}_2(\bR^d,\rho), \mathcal{B}(W^{\tau}_2(\bR^d,\rho)))$, which is just the Borel $\sigma-$field generated by the strong topology on $W^{\tau}_2(\bR^d,\rho)$, see [\ref{Ito}, p. 6].
In this case, the solution $s$ of $(\ref{carma})$ can be identified with a random variable on a weighted Sobolev space, or more generally weighted Besov spaces, too:
\begin{lemma}\label{Besov}
Let  $p,q$ be polynomials in $d$ variables such that there exists $\kappa \in (0,\infty)$ such that 
\begin{align}\label{eqregular1}
\left|D^{\gamma}\frac{q(i\xi)}{p(i\xi)}\right|\le c_\gamma \langle \xi\rangle^{-\kappa-|\gamma|}
\end{align}
for every $\gamma\in \bN^d_0$, where $c_\gamma \ge 0$.
Let $\dot{L}$ be a L\'{e}vy white noise on $\mathcal{S}'$ with characteristic triplet $(a,\gamma,\nu)$ such that $\int_{|r|>1} |r|^\varepsilon \nu(dr)<\infty$ for some $\varepsilon>0$. Let $\rho<-d/\min\{2,\varepsilon\}$, $l<-\frac{d}{2}$ and choose a version of $\dot{L}$ in the Sobolev space $W^l_2(\bR^d,\rho)$ as described above.\\
 Then there exists a solution $s$ of (\ref{carma}) in $\mathcal{S}'$ which almost surely lies in $B^{\tau+\kappa}_{r,r}(\bR^d,\rho)$ whenever $r\in [2,\infty]$ and $\tau\le l+d\left(\frac{1}{r}-\frac{1}{2}\right)$, and even is a random variable in $(B^{\tau+\kappa}_{r,r}(\bR^d,\rho),\mathcal{B}(B^{\tau+\kappa}_{r,r}(\bR^d,\rho)))$.
\end{lemma}
\begin{proof}
By [\ref{Edmunds}, Theorem 5.4.2, p. 224] we conclude that $\varphi\mapsto \mathcal{F}^{-1} \frac{q(i\cdot)}{p(i\cdot)} \mathcal{F}\varphi$ defines a continuous operator both from $\mathcal{S}'$ to $\mathcal{S}'$ and from  $W^{l}_2(\bR^d,\rho)$ to $W^{l+\kappa}_2(\bR^d,\rho)$. We conclude by construction of $s$ in (\ref{construction})  that we have a solution in $\mathcal{S}'$ which is also in $(W^{l+\kappa}_{2}(\bR^d,\rho), \mathcal{B}(W^{l+\kappa}_2(\bR^d,\rho)))$. The rest follows easily by Proposition \ref{embeddings}.
\end{proof}
Observe that if $\kappa>d\left(1-\frac{1}{r}\right)$, with $r\ge2$ we can choose $l$ and $\tau$ from above such that $\tau+\kappa>0$. In this case, $s$ has positive regularity and can be identified with a random field on $\bR^d$ via the mapping of (\ref{mapping}).
\begin{example}
Let $p(D)=(\lambda-\Delta)^{\alpha}$ for $\alpha\in \bN$ and $\lambda>0$ and $q(D)=1$. Then (\ref{eqregular1}) is satisfied for $\kappa=2\alpha$, see [\ref{Grafakos}, Example 6.2.9, p. 449].
\end{example}
\section{Semilinear stochastic partial differential equations}\label{semilinearspdes}
Our goal of this section is to study the semilinear stochastic partial differential equation
\begin{align}\label{eq11}
p(D)s=g(\cdot,s)+\dot{L},
\end{align}
where $\dot{L}$ is a L\'{e}vy white noise on $\mathcal{S}'$ with characteristic triplet $(a,\gamma,\nu)$ such that $\int_{|r|>1}|r|^{\varepsilon}\nu(dr)<\infty$ for some $\varepsilon>0$, $p$ is a polynomial in $d$ variables and $g:\bR^d\times \bC\to \bR$ a sufficiently nice function. We assume that the L\'{e}vy white noise is the modified version on the measurable space $(B^{l}_{2,2}(\bR^d,\rho), \mathcal{B}(B^{l}_{2,2}(\bR^d,\rho)))$, where $l<-d/2$ and $\rho<-\frac{d}{\min\{2,\varepsilon\}}$. We are looking for a $B^\beta_{r,r}(\bR^d,\rho)\subset \mathcal{S}'$-valued solution $s$, where $r\ge 2$ and $\beta>0$. Observe that since $r\ge 2$ and $\beta>0$ every $f\in B^{\beta}_{r,r}(\bR^d,\rho)$ can be identified with a function $\xi(f)\in L^r(\bR^d,\rho)$ in a continuous way via (\ref{mapping}). We again denote by $f$ the function $\xi(f)$. Then $g(\cdot,s)$ means the function
\begin{align*}
g(\cdot,s):\bR^d\to \bR,\, x\mapsto g(x,s(x)),
\end{align*}
which in turn can again be identified with a distribution via
\begin{align*}
\langle g(\cdot,s),\varphi\rangle:=\int\limits_{\bR^d} g(x,s(x))\varphi(x)\,\lambda^d(dx).
\end{align*}
By a $\mathcal{B}_{r,r}^\beta(\bR^d,\rho)\subset \mathcal{S}'$-valued solution of (\ref{eq11}) we mean a measurable mapping $$s:(\Omega,\mathcal{F})\to (B^{\beta}_{r,r}(\bR^d,\rho),\mathcal{B}(B^{\beta}_{r,r}(\bR^d,\rho)))$$ such that
\begin{align*}
\langle s, p(D)^*\varphi \rangle=\int_{\bR^d}g(x,s(x))\varphi(x)\lambda^d(dx)+\langle \dot{L},\varphi\rangle
\end{align*}
for every $\varphi \in \mathcal{S}$.
\begin{proposition}\label{lemma1} Let $r\in[2,\infty]$, $\rho<-\frac{d}{\min\{2,\varepsilon\}}$, $\kappa>d(1-1/r)+\beta$ for some $\beta >0$ and $p(D)$ be a partial differential operator satisfying 
\begin{align}\label{eqregular}
\left|D^{\gamma}\frac{1}{p(i\xi)}\right|\le c_\gamma \langle \xi\rangle^{-\kappa-|\gamma|}
\end{align}
for every $\gamma\in \bN^d_0$, where $c_\gamma \ge 0$. Furthermore, let $g:\bR^d\times \bC\to \bR$ be a Lipschitz function  such that
\begin{align*}
|g(x,y)|\le C(1+|y|)
\end{align*}
for some constant $C>0$ for all $x\in\bR^d$ and $y\in \bC$ and assume that
\begin{align}\label{lip}
\begin{split}
\|g\|_{Lip}:=&\sup_{x\in \bR^d}\sup_{z,y\in \bC} \frac{|g(x,y)-g(x,z)|}{|y-z|}\\
 <&( \| p(D)^{-1}\|_{L^r(\bR^d,\rho)\to B^{\beta}_{r,r}(\bR^d,\rho)}\|id\|_{B^\beta_{r,r}(\bR^d,\rho)\to L^r(\bR^d,\rho)})^{-1}<\infty.
\end{split}
\end{align}
Let $\dot{L}$ be a L\'{e}vy white noise on $\mathcal{S}'$ with characteristic triplet $(a,\gamma,\nu)$ such that $\int_{|r|>1} |r|^\varepsilon \nu(dr)<\infty$. Let $l=\beta-\kappa+d\left(\frac{1}{2}-\frac{1}{r}\right)<-\frac{d}{2}$ and choose a version of $\dot{L}$ in the Sobolev space $B^l_{2,2}(\bR^d,\rho)$ as described above.\\
Then there exists a unique measurable mapping $s:(\Omega,\mathcal{F})\to (B^{\beta}_{r,r}(\bR^d,\rho),\mathcal{B}(B^{\beta}_{r,r}(\bR^d,\rho)))$, which solves the equation (\ref{eq11}).
 Especially, it holds that $s\in L^r(\Omega, B^{\beta}_{r,r}(\bR^d,\rho))$ if $\varepsilon>r\ge 2$.
\end{proposition}
\begin{remark}\label{rembestwaifu}
Observe that 
\begin{align*}
&p(D)^{-1}:B^{\beta-\kappa}_{r,r}(\bR^d,\rho)\to B^{\beta}_{r,r}(\bR^d,\rho)\textrm{ and }\\
&p(D)^{-1}:L^{r}(\bR^d,\rho)\to B^{\beta}_{r,r}(\bR^d,\rho)
\end{align*}
are well-defined and continuous linear operators by assumption (\ref{eqregular}) and that $B^{l}_{r,r}(\bR^d,\rho)$ is continuously embedded into $B^{j}_{r,r}(\bR^d,\rho)$ for every $l>j$ and $L^r(\bR^d,\rho)$ is continuously embedded into $B^{a}_{r,r}(\bR^d,\rho)$ for every $a<0$, see  [\ref{Edmunds}, Theorem 5.4.2, p. 224].
\end{remark}
\begin{proof}
We set $u$ to be the unique solution of 
\begin{align}\label{middle}
p(D)u=\dot{L}.
\end{align}
From Lemma \ref{Besov}  we see that $u$ has a measurable version from $(\Omega,\mathcal{F})$ to $(B^{\beta}_{r,r}(\bR^d,\rho), \mathcal{B}(B^{\beta}_{r,r}(\bR^d,\rho)))$. We see that in order to solve $(\ref{eq11})$ we need to solve the equation
\begin{align}\label{warum2}
p(D)v=g(\cdot,u+v)
\end{align}
with $v\in B^\beta_{r,r}(\bR^d,\rho)$, where $u$ is defined in (\ref{middle}), as $s:=u+v$ solves (\ref{eq11}). Since $u\in B^{\beta}_{r,r}(\bR^d,\rho)$, we have $g(\cdot,u+v)\in L^r(\bR^d,\rho)$, as
\begin{align*}
\int\limits_{\bR^d} \langle x\rangle^{r\rho}|g(x,u(x)+v(x))|^r\lambda^d(dx)&\le \int\limits_{\bR^d}\langle x\rangle^{r\rho}C(1+|u(x)|+|v(x)|)^r\lambda^d(dx)\\
&\le C'(1+\|u\|_{L^r(\bR^d,\rho)}^r+\|v\|_{L^r(\bR^d,\rho)}^r)
\end{align*}
for some suitable constants $C$ and $C'>0$. Moreover, we see from Remark \ref{rembestwaifu} that $p(D)^{-1} g(\cdot,u+v)\in B^\beta_{r,r}(\bR^d,\rho)$.
Therefore let $\tilde{u} \in B^\beta_{r,r}(\bR^d,\rho)$ and we define $\Psi_{\tilde u}=\Psi: B^{\beta}_{r,r}(\bR^d,\rho)\to B^{\beta}_{r,r}(\bR^d,\rho)$ by 
\begin{align*}
\Psi(\varphi)=p(D)^{-1}g(\cdot,\tilde{u}+\varphi)
\end{align*}
for all $\varphi\in B^{\beta}_{r,r}(\bR^d,\rho)$. We show that there exists a fixed point of $\Psi$, which is especially the solution of (\ref{warum2}) for the fixed $\tilde{u}\in {B}^\beta_{r,r}(\bR^d,\rho)$. We see that
\begin{align*}
\| \Psi(\varphi_1)-\Psi(\varphi_2)\|_{B^{\beta}_{r,r}(\bR^d,\rho)}&\le\| p(D)^{-1}\|_{L^r(\bR^d,\rho)\to B^{\beta}_{r,r}(\bR^d,\rho)}\| g(\cdot,\tilde{u}+\varphi_1)-g(\cdot,\tilde{u}+\varphi_2)\|_{L^r(\bR^d,\rho)}\\
&\le\| p(D)^{-1}\|_{L^r(\bR^d,\rho)\to B^{\beta}_{r,r}(\bR^d,\rho)}\|g\|_{Lip}\|id\|_{B^\beta_{r,r}(\bR^d,\rho)\to L^r(\bR^d,\rho)}\|\varphi_1-\varphi_2\|_{B^{\beta}_{r,r}(\bR^d,\rho)}.
\end{align*}
It follows that $\Psi$ is a strict contraction and by Banach's fixed point theorem we conclude that for every $\tilde{u}\in B^\beta_{r,r}(\bR^d,\rho)$ there exists a unique solution $v\in B^\beta_{r,r}(\bR^d,\rho)$ of 
\begin{align*}
p(D) \tilde{v}=g(\cdot,\tilde{u}+\tilde{v}).
\end{align*}
By a small calculation we see that $\tilde{v}$ depends continuously on $\tilde{u}$. Namely let $u_1$ and $u_2$ be in $B^\beta_{r,r}(\bR^d,\rho)$ and let $v_1$ and $v_2$ be the corresponding fixed points. We see that
\begin{align*}
&\|v_1-v_2\|_{B^\beta_{r,r}(\bR^d,\rho)}\\
\le&\| p(D)^{-1}\|_{L^r(\bR^d,\rho)\to B^{\beta}_{r,r}(\bR^d,\rho)}\| g(\cdot,u_1+v_1)-g(\cdot,u_2+v_2)\|_{L^r(\bR^d,\rho)}\\
\le&\| p(D)^{-1}\|_{L^r(\bR^d,\rho)\to B^{\beta}_{r,r}(\bR^d,\rho)}\|g\|_{Lip}\|id\|_{B^\beta_{r,r}(\bR^d,\rho)\to L^r(\bR^d,\rho)}(\|u_1-u_2\|_{B^{\beta}_{r,r}(\bR^d,\rho)}+\|v_1-v_2\|_{B^{\beta}_{r,r}(\bR^d,\rho)}),
\end{align*}
which implies that
\begin{align*}
\|v_1-v_2\|_{B^\beta_{r,r}(\bR^d,\rho)}\le \frac{\| p(D)^{-1}\|_{L^r(\bR^d,\rho)\to B^{\beta}_{r,r}(\bR^d,\rho)}\|g\|_{Lip}\|id\|_{B^\beta_{r,r}(\bR^d,\rho)\to L^r(\bR^d,\rho)}}{1-{\| p(D)^{-1}\|_{L^r(\bR^d,\rho)\to B^{\beta}_{r,r}(\bR^d,\rho)}\|g\|_{Lip}}\|id\|_{B^\beta_{r,r}(\bR^d,\rho)\to L^r(\bR^d,\rho)}}\|u_1-u_2\|_{B^{\beta}_{r,r}(\bR^d,\rho)}.
\end{align*}
We conclude that there exists a measurable solution $v$ of (\ref{warum2}) in the space \\$(B^{\beta}_{r,r}(\bR^d,\rho),\mathcal{B}(B^{\beta}_{r,r}(\bR^d,\rho)))$. As the solution $s$ of (\ref{eq11}) is then given by $s=u+v$, we conclude that we find a unique measurable solution of (\ref{eq11}).\\ 
Now let $\varepsilon>r\ge 2$. We compute as above that
\begin{align*}
\|s\|_{B^{\beta}_{r,r}(\bR^d,\rho)}=&\|p(D)^{-1}p(D)(u+v)\|_{B^{\beta}_{r,r}(\bR^d,\rho)}\\
\leq &\| p(D)^{-1}\|_{L^r(\bR^d,\rho)\to B^{\beta}_{r,r}(\bR^d,\rho)}\|g(\cdot,s)\|_{L^r(\bR^d,\rho)}+ C\|p(D) u\|_{B^{\beta-\kappa}_{r,r}(\bR^d,\rho)}\\
\leq&\| p(D)^{-1}\|_{L^r(\bR^d,\rho)\to B^{\beta}_{r,r}(\bR^d,\rho)}\|g(\cdot,s)\|_{L^r(\bR^d,\rho)}+C\|\dot{L}\|_{B^{\beta-\kappa}_{r,r}(\bR^d,\rho)}
\end{align*}
for some constant $C$, from which we conclude that $$\|s\|_{B^{\beta}_{r,r}(\bR^d,\rho)}\leq C'( 1+\|\dot{L}\|_{B^{\beta-\kappa}_{r,r}(\bR^d,\rho)})$$ for some constant $C'$ and by [\ref{Aziznejad}, Proposition 5] we infer that $\bE \|s\|_{B^{\beta}_{r,r}(\bR^d,\rho)}^{r}<\infty$.
\end{proof}

\begin{example}
Let $d=1,2$ or $3$, $\lambda>0$ and $\dot{L}$ be a L\'{e}vy white noise in $\mathcal{S}'$ with $\int_{|x|>1} |x|^\varepsilon\nu(dx)<\infty$ for some $\varepsilon>0$. Let $\rho<-\frac{d}{\min{2,\varepsilon}}$ and choose a modification of $\dot{L}$ which is $(W^{\alpha}_2(\bR^d,\rho),\mathcal{B}(W^{\alpha}_2(\bR^d,\rho))$ measurable with $\alpha \in (-2,-3/2)$. Then for every $\beta \in (0,2+\alpha]$ there exists some $\varepsilon_{\beta}>0$ such that
\begin{align*}
(\lambda-\Delta)s+c\sin(s)=\dot{L}
\end{align*}
has a unique and measurable solution in $W^\beta_{2}(\bR^d,\rho)$ for all $c\in (0,\varepsilon_{\beta})$.
\end{example}
\begin{proof}
By [\ref{Grafakos}, Example 6.2.9, p. 449] we know that $(\lambda-\Delta)$ satisfies $(\ref{eqregular})$ for $\kappa=2$ and $u$ is even a real-valued distribution. Moreover, we see that $\sin$ is Lipschitz-continuous with Lipschitz constant equal to $1$ and by Proposition \ref{lemma1} we conclude that there exists a unique pathwise solution.
\end{proof}

\section*{Acknowledgement:} This work was funded by DFG grant LI 1026/6-1. Financial support is gratefully acknowledged.

\vspace{1cm}
David Berger\\
 Ulm University, Institute of Mathematical Finance, Helmholtzstra{\ss}e 18, 89081 Ulm,
Germany\\
email: david.berger@uni-ulm.de

\end{document}